\documentclass[11pt]{amsart}

\usepackage{amsmath,amsthm,amssymb}
\usepackage[shortlabels]{enumitem}
\setlist[enumerate,1]{label={\upshape(\roman*)}}
\usepackage{hyperref}
\usepackage{color}

\newtheorem{thm}{Theorem}[section]
\newtheorem{prop}[thm]{Proposition}
\newtheorem{lem}[thm]{Lemma}

\theoremstyle{definition}
\newtheorem{dfn}[thm]{Definition}
\newtheorem{rem}[thm]{Remark}

\DeclareMathOperator{\supp}{supp}

\newcommand{\R}{\mathbb{R}}
\newcommand{\bu}{\boldsymbol{u}}

\newcommand{\nexteq}{\displaybreak[0]\\ &=}
\newcommand{\iantipodal}{symmetric}
\newcommand{\aiantipodal}{a symmetric}
\newcommand{\nantipodal}{non-symmetric}

\title{Antipodality of Spherical Designs with Odd Harmonic Indices}
\author{Ryutaro Misawa}
\address{Graduate School of Information Sciences, Tohoku University}
\email{misawa.ryutaro.q2@dc.tohoku.ac.jp}
\author{Akihiro Munemasa}
\address{Graduate School of Information Sciences, Tohoku University}
\email{munemasa@tohoku.ac.jp}
\author{Masanori Sawa}
\address{Graduate School of System Informatics, Kobe University}
\email{sawa@people.kobe-u.ac.jp}


\date{\today}
\begin{document}

\begin{abstract}
We determine the smallest size of a non-antipodal spherical design with harmonic indices $\{1,3,\dots,2m-1\}$ to be $2m+1$, where $m$ is a positive integer. This is achieved by proving an analogous result for interval designs.
\end{abstract}

\maketitle
%%%%%%%%%%%%%%%%%%%%%%%%%%%%%%%%%%%%%%%%%%%
%%%%%%%%%%%%%%%%%%%%%%%%%%%%%%%%%%%%%%%%%%%
\section{Introduction}

Spherical designs appear in approximation theory, frame theory and combinatorics (see \cite{MRCHEN, MR485471, MR3752185}). A more general concept of spherical designs was introduced by \cite{MR3339051}. A finite subset $X$ of a unit sphere $\mathbb{S}^{d-1}$ in $\R^d$ is defined to be a spherical design of harmonic index $t$ if the sum of the values on $X$ vanishes for every harmonic homogeneous polynomial of degree $t$. If $T$ is a set of positive integers and $X$ is a spherical design of harmonic index $t$ for every $t\in T$, we simply say that $X$ is a spherical $T$-design. Then the ordinary spherical $t$-design is nothing but a spherical $\{1,2,\dots,t\}$-design. For certain finite sets $T$ of even integers, spherical $T$-designs with additional conditions on size or the number of distances have been considered (see, e.g., \cite{MR4662467,MR4080490,MR3434430,MR3650263}).

If $X$ consists of a pair of antipodal points on $\mathbb{S}^{d-1}$, then $X$ is a spherical design of harmonic index $t$ for every odd positive integer $t$. Conversely, it is not difficult to show that if \(X\subset\mathbb{S}^{d-1}\) is a spherical design of harmonic index $t$ for all odd positive integers $t$, then $X$ is necessarily a union of pairs of antipodal points on $\mathbb{S}^{d-1}$. 
For a brief proof of this
fact, suppose that there exists a point $x\in X$ with $-x\notin X$. 
Since $|\langle x,y\rangle|<1$ for all $y\in X\setminus\{x\}$,
\[
\sum_{y\in X}\langle x,y\rangle^{2k-1}\to 1
\]
as $k\to\infty$.
This contradicts the assumption %(fact) 
 that $X$ is a spherical design of 
harmonic index $2k-1$ for all positive integers $k$.

Since a regular $(2m+1)$-gon in $\mathbb{S}^1$ is known to be a spherical $2m$-design (see \cite{MR679209}), it is also a spherical $\{1,3,\dots,2m-1\}$-design. Thus, if $T$ is a finite set of positive integers with $2m-1 =\max T$, then there exists a $(2m+1)$-point spherical $T$-design which is not a union of pairs of antipodal points.

Our main result is to show that $2m+1$ is the smallest size of a finite set $X$ which can be a non-antipodal spherical $\{1,3,\dots,2m-1\}$-design. This is achieved by proving an analogous result for interval designs.

The paper is organized as follows. In Section~\ref{sec:2}, we recall Newton's identities that relate power sums and elementary symmetric polynomials. In Section~\ref{sec:3}, we introduce an analogue of harmonic indices for interval designs and prove a lower bound for the size of a \nantipodal\ design. By regarding this result as a set version, we also prove its function version (i.e., a design with weights). In Section~\ref{sec:4}, we prove our main result on a lower bound for the size of a non-antipodal spherical design with odd harmonic indices. Finally, in Section~\ref{sec:5}, we prove the optimality of our results in Sections~\ref{sec:3} and \ref{sec:4} by constructing non-antipodal designs with the smallest possible sizes.

%%%%%%%%%%%%%%%%%%%%%%%%%%%%%%%%%%%%%%%%%%%
%%%%%%%%%%%%%%%%%%%%%%%%%%%%%%%%%%%%%%%%%%%
\section{Preliminaries}\label{sec:2}

For each non-negative integer \( k\), we denote by \( p_k(x_1, \dots, x_n) \) the \( k \)-th power sum and by \( e_k(x_1, \dots, x_n) \) the \( k \)-th elementary symmetric polynomial in the variables \( x_1, \dots, x_n \). The following fact, widely known as Newton's identity, gives relations between elementary symmetric polynomials and power sums.

\begin{lem}
\label{lem:Newton}
\begin{enumerate}
\item\label{Ni1} 
For all integers \( n \ge k \ge 1 \), we have
\[
k\, e_k(x_1,\ldots,x_n) = \sum_{i=1}^k (-1)^{i-1} e_{k-i}(x_1,\ldots,x_n) \, p_i(x_1,\ldots,x_n).
\]
\item\label{Ni2} 
For all integers \( k \ge n \ge 1 \), we have
\[
0 = \sum_{i=k-n}^k (-1)^{i-1} e_{k-i}(x_1,\ldots,x_n) \, p_i(x_1,\ldots,x_n).
\]
\end{enumerate}
\end{lem}

Since we will be dealing with spherical designs with odd harmonic indices, we establish relationships between power sums of odd degrees and elementary symmetric polynomials of odd degrees.

\begin{prop}
\label{prop:Newton1}
Let \( m \) and \( n \) be positive integers with \( 2m-1 \le n \), and let \( x_1,\ldots,x_n \in \R \). Then the following are equivalent:
\begin{enumerate}
\item\label{N1i1} 
\[
p_{2k-1}(x_1,\ldots,x_n) = 0 \quad \text{for all } k\in\{1,\dots,m\}.
\]
\item\label{N1i2}
\[
e_{2k-1}(x_1,\ldots,x_n) = 0 \quad \text{for all } k\in\{1,\dots,m\}.
\]
\end{enumerate}
\end{prop}
\begin{proof}
The equivalence of \ref{N1i1} and \ref{N1i2} is clear when \( k=1 \). We prove \ref{N1i1} $\Rightarrow$ \ref{N1i2} by induction on \( m \). By Lemma~\ref{lem:Newton}~\ref{Ni1},
\[
e_{2k-1}(x_1,\ldots,x_n) = \frac{1}{2k-1} \sum_{i=1}^{2k-1} (-1)^{i-1} e_{2k-1-i}(x_1,\ldots,x_n) \, p_i(x_1,\ldots,x_n).
\]
Note that in each term \( e_{2k-1-i}\, p_i \) (for \( 1 \le i \le 2k-1 \)), the indices have opposite parity. Hence, by  \ref{N1i1} and the induction hypothesis, we obtain \( e_{2k-1}(x_1,\ldots,x_n)=0 \).

Next, we prove \ref{N1i2} $\Rightarrow$ \ref{N1i1} by induction on \( k \). By Lemma~\ref{lem:Newton}~\ref{Ni1},
\begin{align*}
p_{2k-1}(x_1,\ldots,x_n) &= 
(2k-1)e_{2k-1}(x_1,\ldots,x_n) 
\\&\quad+ \sum_{i=1}^{2k-2} (-1)^{i} e_{2k-1-i}(x_1,\ldots,x_n) \, p_i(x_1,\ldots,x_n).
\end{align*}
Again, in each term \( e_{2k-1-i}\, p_i \) (for \( 1 \le i \le 2k-2 \)), the indices have different parity. Then by \ref{N1i2} and the induction hypothesis, we have \( p_{2k-1}(x_1,\ldots,x_n)=0 \).
\end{proof}

%%%%%%%%%%%%%%%%%%%%%%%%%%%%%%%%%%%%%%%%%%%
%%%%%%%%%%%%%%%%%%%%%%%%%%%%%%%%%%%%%%%%%%%
\section{Interval Designs}\label{sec:3}

Let $(a,b)$ be an open interval in $\mathbb{R}$ and let $w(x)$ be a probability density function such that all moments $\int_a^b x^s \, w(x)dx$ are finite. Let $x_1,\ldots,x_n \in (a,b)$. In numerical analysis, an integration formula of the type
\[
\frac{1}{n}\sum_{k=1}^{n} x_k^s = \int_a^b x^s \, w(x)dx,
\]
for \( s=1,2,\ldots,t \), is called a \emph{strict Chebyshev-type quadrature of degree \( t \)} if all \( x_i \) are distinct, and a \emph{Chebyshev-type quadrature of degree \( t \)} otherwise~\cite{MR1240246}. As exemplified by the Chebyshev-Gauss quadrature
\[
\frac{1}{n}\sum_{k=1}^{n}\Bigl(\cos\Bigl(\frac{2k\pi}{2n-1}\Bigr)\Bigr)^s = \int_{-1}^{1} x^s \, \frac{1}{\pi\sqrt{1-x^2}}\,dx,
\]
where \( s \) ranges over \( 1,2,\dots,2n-1 \), Chebyshev-type formulas for the ultraspherical measure
\(
(1-x^2)^{\lambda-1/2}\,dx/ \int_{-1}^1 (1-x^2)^{\lambda-1/2}\,dx
\), where $\lambda$ is a nonnegative integer, are a central object in the study of quadrature formulas. Meanwhile, the study of the configuration of points \( x_1,\dots,x_n \) has long captured the attention of researchers in combinatorics and related areas, where the term ``interval design'' is widely used for the point configuration; for example see Bajnok~\cite{MR1129808}.

\begin{dfn}
\label{def:intt-design}
A subset \( X = \{x_1, x_2, \dots, x_n\} \) of the interval \([a,b]\) is called an \emph{interval \( t \)-design} if
\begin{equation}\label{interval_design}
\frac{1}{n}\sum_{k=1}^{n} x_k^s = \frac{1}{b-a}\int_a^b x^s\,dx
\end{equation}
holds for all \( s=1,2,\dots,t \). More generally, for \( T\subset\mathbb{N} \), the set \( X \) is called an interval design of \emph{index} \( T \) if \eqref{interval_design} holds for every \( s\in T \).

We will only consider interval designs of index \( T_m \) or interval \( T_m \)-designs, for brevity on the interval \([-1,1]\), where
\begin{equation}\label{Tm}
T_m=\{2k-1\mid 1\le k\le m\}.
\end{equation}
In particular, the definition can be stated for any measure on \([-1,1]\) by replacing the right-hand side of \eqref{interval_design} with the corresponding moment integral. 
It may also be worth mentioning that \eqref{interval_design} is equivalently written as \(p_s(x_1,\dots,x_n)=0\) for an odd positive integer \(s\).

We say that a finite subset \( X\subset[-1,1] \) is \emph{\iantipodal} if \( X = -X \). More generally, a multiset \( X\subset[-1,1] \) is said to be \emph{\iantipodal} if the multiplicity of \( x \) in \( X \) coincides with that of \( -x \) for every \( x\in X \).
\end{dfn}

\begin{prop}
\label{prop:antipodal2}
Let \( m \) and \( n \) be positive integers with \( n\le 2m \). Suppose that real numbers \( x_1,\dots,x_n\in[-1,1] \) satisfy
\begin{equation}\label{xi2k-1}
p_{2k-1}(x_1,\dots,x_n)=0
\end{equation}
for all \( k\in\{1,\dots,m\} \). Then the multiset \(\{x_1,\dots,x_n\}\) is \iantipodal. In particular, if a subset \( X\subset[-1,1] \) is an interval \( T_m \)-design with \(|X|=n\), 
then \( X \) is \iantipodal.
\end{prop}
\begin{proof}
Without loss of generality, we may assume \(0\notin X\), because removing zeros does not affect the conditions \(\sum_{x\in X} x^{2k-1}=0\) for \(1\le k\le m\).
We first prove that \eqref{xi2k-1} holds for every positive integer \( k \) by induction on \( k \). Since~\ref{xi2k-1} is assumed for \( k\le m \), so suppose \( k>m \). Then \( 2k-1\ge n \). 
Moreover, because there are only \( n \) variables, \( e_j(x_1,\dots,x_n)=0 \) for all \( j>n \); in particular, for \( j>2m \). 
It follows from Lemma~\ref{lem:Newton}~\ref{Ni2} that
\begin{align*}
p_{2k-1}(x_1,\dots,x_n) &= \sum_{i=2k-1-n}^{2k-2} (-1)^i e_{2k-1-i}(x_1,\dots,x_n) \, p_i(x_1,\dots,x_n)\\
&=\sum_{i=1}^{m} e_{2i-1}(x_1,\dots,x_n) \, p_{2(k-i)}(x_1,\dots,x_n)\\
&\quad-\sum_{i=1}^{m} e_{2i}(x_1,\dots,x_n) \, p_{2(k-i)-1}(x_1,\dots,x_n).
\end{align*}

By Proposition~\ref{prop:Newton1} and \eqref{xi2k-1}, 
we have $e_{2i-1}(x_1,\dots,x_n)=0$ for $1\leq i\leq m$.
Thus, the first term is zero.
By the inductive hypothesis,
the second term
%% the second terms
is also zero.
Therefore, we have shown that
\eqref{xi2k-1} holds for every positive integer $k$.

We now prove the assertion by induction on $n$.
The cases $n=1,2$ are obvious.
Let $Y=\{y_1,\dots,y_{n'}\}$ be the set of all distinct elements of 
%in the special case where 
$x_1,\dots,x_n$, that is, $Y=\{x_1,\dots,x_n\}$ and
$|Y|=n'\leq n$.
%are pairwise distinct. We use induction on $n$.
Define $A\in\R^{n'\times n'}$ by
\[
A_{i,j}=y_j^{2i-1}\quad(1\leq i,j\leq n').
\]
Define a column vector $\bu$ whose $j$-th entry $u_j$ is given by
\[
u_j=|\{i\mid 1\leq i\leq n,\;x_i=y_j\}|\quad(1\leq j\leq n').
\]
Then the $k$-th entry of $A\bu$ is
\[
\sum_{j=1}^{n'} y_j^{2k-1}u_j=p_{2k-1}(x_1,\dots,x_n)=0,
\]
since we have shown that \eqref{xi2k-1} holds for every positive integer $k$.
Thus $A\bu=0$, and hence
%By the assumption, we have $A\bu=0$, where
%$\bu$ is the column vector whose $j$-th entry is 
%and hence
\[
0=\det A=\prod_{k=1}^{n'}y_{k}\prod_{1\leq i<j\leq n'}(y_j^2-y_i^2).
\]
This implies that, either there exists $j$ such that $y_{j}=0$, or
there exist $i,j$ with $i<j$ such that $y_i=-y_j$.
In the latter case, we may assume without loss of generality $x_{n-1}=-x_n$.
%, and
%$X'=X\setminus\{x_i,x_j\}$ in the latter case.
Since we have shown that \eqref{xi2k-1} holds for every positive integer $k$,
we obtain
\[p_{2k-1}(x_1,\dots,x_{n-2})=p_{2k-1}(x_1,\dots,x_{n})=0\]
for every positive integer $k$.
%\sum_{x\in X'}x^{2k-1}=\sum_{x\in X}x^{2k-1}=0
%\quad(1\leq k\leq n),\]
%we have, in particular,
%\[\sum_{x\in X'}x^{2k-1}=\sum_{x\in X}x^{2k-1}=0\quad(1\leq k\leq |X'|).\]
By induction, the multiset $\{x_1,\dots,x_{n-2}\}$ is \iantipodal,
and so is $\{x_1,\dots,x_{n}\}$.
The former case can be proved in a similar manner.
\end{proof}
\begin{rem}
Although in this paper we restrict ourselves to interval designs for the normalized Lebesgue measure on $[-1,1]$, the notion of an interval design naturally extends to an arbitrary measure on an interval. In this broader setting, a challenging problem is the construction of interval designs with rational points and weights for the  ultraspherical measure (see Bannai et al.~\cite[p.~208, Problem~2]{BBIT2019}). It was Mishima et al.~\cite{MLSU2024} who first observed Proposition~\ref{prop:antipodal2} for $m \le 3$ and thereby established, together with explicit constructions, a classification theorem of rational $5$-designs with $6$ points for the Chebyshev measure $(1-t^2)^{-1/2}dt/\pi$.
\end{rem}
%% Mishima et al.~\cite{MLSU2024} first observed Proposition~\ref{prop:antipodal2} for $m \le 3$ and then, independently of the present work, obtained Proposition~\ref{prop:antipodal2} in general. Our proof of Proposition~\ref{prop:antipodal2} (and Proposition~\ref{prop:Newton1}) directly makes use of Newton's identity, whereas their proof does not.
%% Mishima et al.~\cite{MLSU2024} で Proposition~\ref{prop:antipodal2} の別証明を記載せず，本論文を引用することになりました。  MEMO1 by Sawa2025/7/7

In Proposition~\ref{prop:antipodal2}, the set 
$\{x_1,\dots,x_{n}\}$ is regarded as a multiset.
This is equivalent to consider a function $\lambda$ on
$\{x_1,\dots,x_{n}\}$ taking values in the set of positive integers.
If we allow $\lambda$ to be a real-valued function, we can also conclude
$\lambda$ to be an even function under a stronger hypothesis.

\begin{dfn}\label{def:w-intt-design}
Let \( m \) be a positive integer and let \( \lambda \colon[a,b]\to\R \) be a function with finite support \( X\subset[a,b] \). The function \( \lambda \) is called a \emph{weighted interval design of index \( T \)} if
\begin{equation}\label{w_interval_design}
\sum_{x\in X} x^s\,\lambda(x)=\frac{1}{b-a}\int_a^b x^s\,dx
\end{equation}
holds for all \( s\in T \). The nonzero values \( \lambda(x) \) are called \emph{weights}.
\end{dfn}

Similar to interval designs, we will only
consider weighted interval designs of index $T_m$, or
weighted interval $T_m$-design for brevity, on the interval $[-1,1]$.

We say that a function $\lambda : [-1,1] \to \R$ with finite support  is 
\emph{\iantipodal} if $\lambda(x)=\lambda(-x)$ holds for all $x\in[-1,1]$.
The same terminology is used in~\cite[Section~6]{MR1240246} in the context of the theory of quadrature formulas.

\begin{lem}\label{lem: vec}
The set of all weighted interval designs of index $T_m$ forms a vector space $D_m$ which contains every symmetric function $\lambda : [-1,1] \to \R$ with finite support.
\end{lem}

\begin{proof}
This follows immediately from the definition.
\end{proof}

\begin{prop}\label{prop:f}
Let \( m \) and \( n \) be positive integers with \( n\le m \). Let \( \lambda \colon\R\to\R \) be a function with finite support \( X\not\ni 0 \) of size \( |X|=n \).  Suppose that \( \lambda \) is a weighted interval \( T_m \)-design. Then $n$ is even and \( \lambda \) is \iantipodal.
\end{prop}
%%%%%%%%%%%%%%%%
%new proof
\begin{proof}

We prove the assertion by induction on $n$.
If $n=1$, then it is clear that there is no such 
weighted interval $T_m$-design for $m\geq1$. 
Hence, we may suppose $n\ge 2$.
Let $X=\{x_1,\dots,x_n\}$,
and define $A\in\R^{n\times n}$ by
\[A_{i,j}=x_j^{2i-1}\quad(1\leq i,j\leq n).\]
Define $v\in\R^n$ by 
\[v_j=\lambda(x_j)\quad(1\leq j\leq n).\]
Then for $1\leq i\leq n$,
\begin{align*}
(Av)_i&=
%\sum_{j\in[n]} A_{i,j}v_j\\&=
\sum_{j=1}^n x_j^{2i-1} \lambda(x_j) \\
%%% \sum_{j\in[n]} x_j^{2i-1} f(x_j) \\
%%% [n] = \{ 1, 2, \ldots, n \} or \{0, 1, \ldots, n \}?
&=
\sum_{x\in\R} x^{2i-1} \lambda(x)
\\&=
0.
\end{align*}
This implies
\[
0=\det A =
\prod_{k=1}^n x_{k}\prod_{1\leq i<j\leq n}(x_j^2-x_i^2).
\]
Thus, 
%there exists $i\in[n]$ such that $x_i=0$, or
there exist $i,j$ with $1\leq i<j\leq n$ such that $x_j=-x_i$.
%If $n=1$, then we have $X=\{0\}$, and the assertion clearly holds.
We may assume without loss of generality $x_n=-x_{n-1}$.

Suppose $n=2$. 
Then 
%If there exists $i\in\{1,2\}$ such that $x_i=0$, then
%we may assume without loss of generality $x_1=0\neq x_2$. Then \eqref{1}
%with $k=1$ implies $x_2f(x_2)=0$ which is absurd. Thus, 
%we have 
$x_2=-x_1$, and hence $x_1(\lambda(x_1)-\lambda(x_2))=0$
%By \eqref{1} with $k=1$, we have $x_1(f(x_1)-f(-x_1))=0$ 
which implies
$\lambda(x_1)=\lambda(-x_1)$. Since the support of $\lambda$ is $\{x_1,x_2\}=\{x_1,-x_1\}$, 
we conclude $\lambda(x)=\lambda(-x)$ for all $x\in\R$.

Now suppose $n\geq3$.
%The cases $n=1,2$ are obvious.
%By the assumption, we have $A\allone=0$, and hence
Define a function $\lambda' :=\lambda-\omega$, where
\[
\omega(x)=
\begin{cases}
\lambda(x_{n})&\text{if $x=x_{n}$,}\\
-\lambda(x_{n})&\text{if $x=-x_{n}$,}\\
0&\text{otherwise.}\\
\end{cases}
\]
Then $\omega$ is symmetric with finite support, and the support of $\lambda'$ has size either $n-1$ or $n-2$.
Since $n \le m$, we have $\lambda \in D_m \subset D_{n-1}$. By Lemma~\ref{lem: vec}, it follows that $\omega \in D_{n-1}$, and $\lambda'=\lambda-\omega \in D_{n-1}$. Then, the inductive hypothesis implies $\lambda'$ is symmetric. Since $\lambda=\lambda'+\omega$, $\lambda$ is also symmetric. In particular, $n$ is even.
\end{proof} 

%%%%%%%%%%%%%%%%%%%%%%%%%%%%%%%%%%%%%%%%%%%
%%%%%%%%%%%%%%%%%%%%%%%%%%%%%%%%%%%%%%%%%%%
\section{Spherical Designs}\label{sec:4}

For \( a,b\in\R^d \), we denote by \( \langle a,b\rangle \) the standard inner product.

\begin{dfn}
\label{def:spht-design}
For a positive integer \( t \), a finite nonempty subset \( X\subset\mathbb{S}^{d-1} \) is called a \emph{spherical \( t \)-design} if
\[
\sum_{x\in X} \langle x,a\rangle^k =
\begin{cases}
\displaystyle \frac{1\cdot3\cdots(2k-1) \langle a,a\rangle^{\frac{k}{2}} }{(d+2)(d+4)\cdots(d+2k-2) }\, & \text{if } k\in\{2,4,\ldots, 2 \lfloor \frac{t}{2} \rfloor \},\\[2mm]
0, & \text{if } k\in\{1,3,\ldots,2 \lfloor \frac{t-1}{2} \rfloor+1 \},
\end{cases}
\]
for all \( a\in\R^d \).
\end{dfn}

Among various definitions of spherical designs
(See \cite[Theorem 2.2]{MR2535394}),
%%% (See \cite[Theorem2.2]{MR2535394}),
the above definition, which originally goes back to Hilbert's solution of Waring problem, can be found in Lyubich and Vaserstein~\cite{MR1235223}, Reznick~\cite{MR1096187}.
Venkov~\cite[Section 3]{MR1881618} also shows that the above definition is equivalent to the standard one in Delsarte, Goethals and Seidel ~\cite{MR485471}.
In fact, denote by $|\mathbb{S}^{d-1}|$ the surface area $\int_{\mathbb{S}^{d-1}}1\,d\rho$
of the sphere $\mathbb{S}^{d-1}$.
It follows, by the $O(d)$-invariance of the uniform measure $\rho$ on $\mathbb{S}^{d-1}$, that
\begin{align*}
&\frac{1}{|\mathbb{S}^{d-1}|} \int_{\mathbb{S}^{d-1}} \langle y,a \rangle^{2k} d\rho
=
\frac{1}{|\mathbb{S}^{d-1}|} \int_{\mathbb{S}^{d-1}} \langle y,\frac{a}{\|a\|} \rangle^{2k} \langle a,a \rangle^k d\rho
\\
&=
\frac{\langle a,a \rangle^k}{|\mathbb{S}^{d-1}|} \int_{\mathbb{S}^{d-1}} y_1^{2k} d\rho 
=
\frac{1 \cdot 3 \cdots (2k-1)}{d(d+2)\cdots(d+2k-2)} \langle a,a \rangle^k,
\end{align*}
which provides an interpretation of the right-hand side of the equation in 
Definition~\ref{def:spht-design} in an analytic manner. 
For a brief explanation on the relation between spherical designs and Hilbert's work, we refer the reader to Seidel~\cite{MR1313290}.

The following more general concept of spherical designs was
introduced by \cite{MR3339051}, and in its full generality by 
\cite{MR3004471}.
See also \cite{MR4662467,MR4080490,MR3650263}.

\begin{dfn}
\label{def:sphT-design}
Let \( Q_{d,t}(s) \) denote the Gegenbauer polynomial of degree \( t \) on \(\mathbb{S}^{d-1}\) (see \cite[Lemma 2.1]{MR3339051}). Let \( T \) be a set of positive integers. A finite nonempty subset \( X\subset\mathbb{S}^{d-1} \) is called a spherical design of \emph{harmonic index} \( T \) if
\[
\sum_{x,y\in X} Q_{d,t}(\langle x,y\rangle)=0 \quad \text{for all } t\in T.
\]
\end{dfn}

We shall only consider spherical designs of harmonic index \( T_m \) (or spherical \( T_m \)-designs), where \( T_m \) is defined in \eqref{Tm}. It is clear that any pair of antipodal points forms a spherical \( T_m \)-design for every positive integer \( m \). Thus, if \( X \) is a spherical \( T_m \)-design containing such a pair \( \{a,-a\} \), then \( X\setminus\{a,-a\} \) is also a spherical \( T_m \)-design.

\begin{lem}
\label{lem:sphTm-design}
Let \( X \) be a finite subset of \(\mathbb{S}^{d-1}\). Then \( X \) is a spherical \( T_m \)-design if and only if $\{\langle x,a \rangle\}_{x\in X}$ is an interval $T_m$-design for all $a \in \mathbb{S}^{d-1}$.

%\[
%\sum_{x\in X} \langle x,a\rangle^k=0\quad\text{for all } a\in\R^d \text{ and all } k\in T_m.
%\]
\end{lem}
\begin{proof}
Considering the univariate polynomial ring \( \R[s] \), note that the span of the monomials \( s^t \) for \( t\in T_m \) coincides with that of the Gegenbauer polynomials \( Q_{d,t}(s) \) (since \( Q_{d,t}(s) \) is odd of degree \( t \) when \( t\in T_m \)). This equivalence establishes the claim.
\end{proof}

\begin{thm}
\label{thm:sph_main}
Let \( X \) be a spherical \( T_m \)-design consisting of \( n \) points. If \( n\le 2m \), then \( X \) is antipodal.
\end{thm}
\begin{proof}
Suppose \( X \) is not antipodal. Then there exists some \( a\in X \) such that \( -a\notin X \). 
The projection of \( X \) onto the line determined by \( a \) yields a multiset of \( n \) real numbers that is non-symmetric. By Proposition~\ref{prop:antipodal2}, such a configuration forces \( n>2m \), contradicting the assumption.
\end{proof}

%%%%%%%%%%%%%%%%%%%%%%%%%%%%%%%%%%%%%%%%%%%
%%%%%%%%%%%%%%%%%%%%%%%%%%%%%%%%%%%%%%%%%%%
\section{Optimality}\label{sec:5}

In this section, we show that the inequalities in the assumptions of
Propositions~\ref{prop:antipodal2}, \ref{prop:f} and
Theorem~\ref{thm:sph_main} cannot be weakened.
First, we show that the inequality $n\le2m$ in 
Proposition~\ref{prop:antipodal2} cannot be weakened.
Recall that we denote by \( p_k \) and \( e_k \) 
the \( k \)-th power sum and $k$-th elementary symmetric polynomial, respectively
with $n$ variables.
For a finite set $A=\{a_1,\dots,a_n\}$ with $n$ elements, we write
$e_k(a_1,\dots,a_n)=e_k(A)$ and $p_k(a_1,\dots,a_n)=p_k(A)$.

\begin{prop}\label{prop:2m+1}
Let $m$ be a positive integer. Then for every positive integer $n$ with
$n>2m$, there exists an interval $T_m$-design $X\subset[-1,1]$ with 
$|X|=n$ and $X\neq-X$. In particular, $2m+1$ is the smallest size
of a \nantipodal\ interval $T_m$-design.
\end{prop}
\begin{proof}
It suffices to prove the case where $n=2m+1$.
Indeed, a set $X$ with $|X|=2m+1$ satisfying the conclusion must
satisfy the stronger property $X\cap(-X)=\emptyset$, by
Proposition~\ref{prop:antipodal2}. Thus, the case
$n=2m+2$ follows by adding $0$ to $X$. For $n>2m+2$,
we can construct a desired $T_m$-design by adding
$\left\lfloor{(n-2m)/2}\right\rfloor$ pairs of antipodal points
in $(-1,1)$ to $X$ with $|X|=2m+1$ or $2m+2$.

Let
\[A=\left\{ \pm \frac{2k-1}{2m} \mid 1\leq k\leq m\right\}.\]
Let $f(x)$ be the monic polynomial of degree $2m$ having
$A$ as its set of roots.
We can choose a small $\varepsilon>0$ in such a way that
the polynomial $g(x)=f(x)+\varepsilon$ has $2m$ roots
in $(-1+\frac{1}{2m},1-\frac{1}{2m})\setminus\{\frac{1}{2m}\}$.
Let $A'$ be the set of roots of $g(x)$. Then, {by the relation between roots and coefficients,}
\[
  e_k(A)=e_k(A') \quad (0\le k\le 2m-1).
\]
By Lemma~\ref{lem:Newton} using induction on $k$, we find
\begin{equation}\label{eq:AA'}
p_k(A)=p_k(A')\quad(0\leq k\leq 2m-1).
\end{equation}

Since $A$ and $(A-\tfrac{1}{2m})\cup\{1\}$ are both \iantipodal,
$p_{\mathrm{odd}}(A)=0$, and
\begin{align*}
0&=p_{2k+1}((A-\frac{1}{2m})\cup\{1\})
\nexteq
1+\sum_{a\in A}(a-\frac{1}{2m})^{2k+1}
\nexteq
1+\sum_{a\in A}\sum_{\ell=0}^{2k+1}
\binom{2k+1}{\ell}a^\ell(-\frac{1}{2m})^{2k+1-\ell}
\nexteq
1+\sum_{\ell=0}^{2k+1}
\binom{2k+1}{\ell}(-\frac{1}{2m})^{2k+1-\ell}
p_\ell(A)
\nexteq
1+\sum_{\ell=0}^{k}
\binom{2k+1}{2\ell}(-\frac{1}{2m})^{2(k-\ell)+1}
p_{2\ell}(A)
\nexteq
1-\sum_{\ell=0}^{k}
\binom{2k+1}{2\ell}\frac{1}{(2m)^{2(k-\ell)+1}}
p_{2\ell}(A).
\end{align*}
Thus, by \eqref{eq:AA'}, we obtain
\[1-\sum_{\ell=0}^{k}
\binom{2k+1}{2\ell}\frac{1}{(2m)^{2(k-\ell)+1}}
p_{2\ell}(A')=0\quad(0\leq k\leq m-1).
\]
Since $A'$ is also \iantipodal, we can 
reverse the above calculation with $A$ replaced by $A'$, and we
conclude
\[p_{2k+1}((A'-\frac{1}{2m})\cup\{1\})=0
\quad(0\leq k\leq m-1).\]
This implies that the set
$X=(A'-\frac{1}{2m})\cup\{1\}$ is an interval $T_m$-design
with $|X|=2m+1$. Since $1\in X$ and $-1\notin X$, we have
$X\neq(-X)$. 

The second assertion then follows from Proposition~\ref{prop:antipodal2}.
\end{proof}

{Next, we show that the inequalities in the assumptions of
Proposition~\ref{prop:f} cannot be weakened.}
The following {two} propositions {show} that $m+1$ is the smallest size of the support
of a \nantipodal\ weighted $T_m$-design. 

\begin{prop}\label{prop:f_opt+}
Let {$n$} be a positive integer.
Define the function $\lambda\colon[-1,1]\to\R$ by
\[
\lambda(x)=
\begin{cases}
2, & \text{if } x = \cos\Bigl(\dfrac{2j\pi}{2{n}+1}\Bigr) \text{ for } 1\le j\le {n}, \\[1mm]
1, & \text{if } x = 1, \\[1mm]
0, & \text{otherwise.}
\end{cases}
\]
Then, $\lambda$ is a weighted interval $T_{{n}-1}$-design.

% for each integer $l$ with $1\le l\le s$, it holds that
%\[
%\sum_{x\in\supp(f)} x^{2l-1} f(x)=0.
%\]
 
\end{prop}

\begin{proof}
For $1\le l \le n$,
\begin{align*}
\sum_{x\in\supp(f)} x^{2l-1}\lambda(x) 
&= \sum_{j=1}^{n} 2\cos^{2l-1}\Bigl(\frac{2j\pi}{2n+1}\Bigr)+ 1    \\
&= \sum_{j=0}^{2n}\cos^{2l-1}\Bigl(\frac{2j\pi}{2n+1}\Bigr)=0,
\end{align*}
since the vertices of the regular $(2n+1)$-gon form a spherical $2n$-design.
\end{proof}

In this way, we always obtain a function \( \lambda \) with support size \( n \) that is a weighted interval \( T_{n-1} \)-design. This shows that the inequality \( n\leq m \) in the assumption of Proposition~\ref{prop:f} cannot be relaxed. 

Using a combinatorial lemma, we can also construct a non-symmetric weighted \(T_{n-1}\)-design on \([-1,1]\) supported on \(n\) rational nodes with rational weights. In its proof, we
denote by \( \delta_{ij} \) the Kronecker's delta symbol.

\begin{lem}\label{lem:binom}
For positive integers $n$ and $s$ with $0\leq s<n$, we have
\[\sum_{j=1}^n (-1)^j j^{2s}\binom{2n}{n-j}=
\begin{cases}
-\binom{2n-1}{n-1}&\text{if $s=0$,}\\
0&\text{if $1\leq s<n$.}
\end{cases}
\]
\end{lem}
\begin{proof}
We prove the assertion by induction on $s$.

Suppose $s=0$, then for all $n>0$ we have:
\begin{align*}
\sum_{j=1}^n (-1)^j \binom{2n}{n-j}&=
\frac12\sum_{j=1}^n (-1)^j\left(\binom{2n}{n-j}+\binom{2n}{n+j}\right)
\nexteq
\frac{(-1)^n}{2}\sum_{j=0}^{2n}(-1)^j\binom{2n}{j}-\frac12\binom{2n}{n}
%\sum_{j=1}^n (-1)^j \left(\binom{2n-1}{n-j}+\binom{2n-1}{n-j-1}\right)
%\nexteq
%\sum_{j=1}^n (-1)^j \binom{2n-1}{n-j}+
%\sum_{j=1}^{n-1} (-1)^j \binom{2n-1}{n-j-1}
%\nexteq
%-\binom{2n-1}{n-1}+
%\sum_{j=2}^n (-1)^j \binom{2n-1}{n-j}+
%\sum_{j=2}^{n} (-1)^{j+1} \binom{2n-1}{n-j}
\nexteq
-\binom{2n-1}{n-1}.
\end{align*}

Now let $s\geq1$, and assume that the assertion holds for $s-1$ and  for all $n >s-1$. Since
\begin{align*}
j^2\binom{2n}{n-j}
&=n^2\binom{2n}{n-j}-2n(2n-1)\binom{2n-2}{n-1-j}
\quad(1\leq j\leq n-1),
\end{align*}
we have for $n>s$,
%Then
\begin{align*}
& \sum_{j=1}^{n} (-1)^j j^{2s}\binom{2n}{n-j} \\
& =(-1)^{n}n^{2s}+ \sum_{j=1}^{n-1} (-1)^j j^{2s}\binom{2n}{n-j} \nexteq (-1)^{n}n^{2s}+ \sum_{j=1}^{n-1} (-1)^j j^{2s-2} n^2\binom{2n}{n-j} \\
& \quad-2n(2n-1)\sum_{j=1}^{n-1} (-1)^j j^{2s-2}\binom{2n-2}{n-1-j} \nexteq
n^2\sum_{j=1}^{n} (-1)^j j^{2(s-1)} \binom{2n}{n-j} -2n(2n-1)\sum_{j=1}^{n-1} (-1)^j j^{2(s-1)}\binom{2(n-1)}{n-1-j}.
\end{align*}
By the inductive hypothesis, we have
\begin{align*}
\sum_{j=1}^{n} (-1)^j j^{2s}\binom{2n}{n-j}
&=-n^2\delta_{s,1}\binom{2n-1}{n-1}
+2n(2n-1)\delta_{s,1}\binom{2n-3}{n-2}=0.
\end{align*}
\end{proof}

%If $m$ is odd, then we can construct a non-symmetric weighted interval 
%$T_m$-design with support consisting  of $m+1$ rational numbers as follows:

{
\begin{prop}\label{prop:f_opt2+}
Let $n$ be a positive integer.
Define $\lambda\colon[-1,1]\to\R$ by
\[\lambda(x)=\begin{cases}
2j\binom{2n}{n-2j}&
\text{if $x=\frac{2j}{n+1}$ for $1\leq j\leq \left\lfloor n/2\right\rfloor$,}\\
(2j-1)\binom{2n}{n-2j+1}&
\text{if $x=-\frac{2j-1}{n+1}$ for $1\leq j\leq \left\lceil n/2\right\rceil$,}\\
%\frac{2j}{n}\binom{2n}{n-2j}&\text{if $x=\frac{2j}{n+2}$ for $1\leq j\leq n/2$,}\\
%\frac{2j-1}{n}\binom{2n}{n-2j+1}&\text{if $x=-\frac{2j-1}{n+2}$ for $1\leq j\leq n/2$,}\\
0&\text{otherwise.}
\end{cases}
\]
Then $\lambda$ is a weighted interval $T_{n-1}$-design.
\end{prop}
\begin{proof}
For $1\leq k\leq n-1$, we have
\begin{align*}
&\sum_{x\in\R} x^{2k-1}\lambda(x) 	 
\\&=
\frac{1}{(n+1)^{2k-1}}\left(  \sum_{j=1}^{\left\lfloor n/2\right\rfloor} (2j)^{2k} \binom{2n}{n-2j}
- \sum_{j=1}^{\left\lceil n/2\right\rceil} (2j-1)^{2k} \binom{2n}{n-2j+1} \right)
%\nexteq
%\frac{1}{n(n+2)^{2k-1}}\left(  \sum_{\substack{j=1\\j\equiv0[2]}}^{n} j^{2k} \binom{2n}{n-j} 
%- \sum_{\substack{j=1\\j\equiv1[2]}}^{n} j^{2k} \binom{2n}{n-j} \right)
\nexteq
\frac{1}{(n+1)^{2k-1}} \sum_{j=1}^{n} (-1)^j j^{2k} \binom{2n}{n-j}
\nexteq
0,
\end{align*}
by Lemma~\ref{lem:binom}.
\end{proof}
}

%If we allow negative values for the function $f$, 
%the construction is slightly easier.

%\begin{prop}\label{prop:f_opt}
%Let $n$ be a positive integer.
%Define $f\colon\R\to\R$ by
%\[f(x)=\begin{cases}
%(-1)^j\binom{2n}{n-j}
%&\text{if $x=j^2/n^2$ for $1\leq j\leq n$,}\\
%0&\text{otherwise.}
%\end{cases}
%\]
%Then $f$ is a weighted interval $T_{n-1}$-design.
%\end{prop}
%\begin{proof}
% For $1\leq k\leq n-1$, we have
%	\begin{align*}
%		 \sum_{x\in\R} x^{2k-1}f(x) 	 
%&=
%		\frac{1}{n^{4k-2}} \left(  \sum_{j=1}^{n} (-1)^j j^{2(2k-1)} \binom{2n}{n-2j}\right)
%\nexteq
%0,
%	 \end{align*}
%Therefore, 
%by Lemma~\ref{lem:binom}.
%\end{proof}

Finally,  we show that the inequality $n\le2m$ in 
Theorem~\ref{thm:sph_main} cannot be weakened.
We need some preparations.

%Let $\Hom_l (s)$ denote the vector space of homogeneous polynomials of degree $l$.
%For positive integers $d'>d$, we regard $\mathbb{S}^{d-1}$ as a subset of $\mathbb{S}^{d'-1}$
%by adding $d'-d$ coordinates of $0$.

%Let $\Hom_l (\mathbb{S}^{d-1})$ denote the vector space of homogeneous polynomials of degree $l$ on $\mathbb{S}^{d-1}$. Furthermore, "embedding a finite subset of $\mathbb{S}^{d-1}$ into $\mathbb{S}^d$" refers to embedding the finite subset onto an orthodrome of $\mathbb{S}^d$.

\begin{lem}
\label{lem:embedding}
Let $m$ be a positive integer.
If $X$ is a spherical \( T_m \)-design on \(\mathbb{S}^{d-1} \),
then \( X \) is also a spherical \( T_m \)-design on \(\mathbb{S}^{d'-1} \)
for \( d' > d \).
\end{lem}

\begin{proof}
This is immediate from Lemma~\ref{lem:sphTm-design}, since
the inner product $\langle x,a\rangle$ for vectors $x,a\in\R^d$
remains the same after regarding them as vectors in $\R^{d'}$.
%Suppose that a spherical $T_m$-design on $\mathbb{S}^{d-1}$, say $X$, is
%embedded in $\mathbb{S}^{d}$.
%Then our goal is to show that
%\begin{equation}\label{eq:design1}
%\int_{\mathbb{S}^d} f(\xi) d\xi^{(d)} = 0 = \sum_{x \in X} f(x)
%\ \text{ for every $l \in T_m$ and $f(\xi) = \xi_1^{c_1} \cdots
%\xi_d^{c_d} \xi_{d+1}^{c_{d+1}} \in {\rm Hom}_l(\mathbb{S}^d)$},
%\end{equation}
%where $d\xi^{(d)}$ denotes the surface measure on $\mathbb{S}^d$.
%Since at least one of $c_i$ is odd, we have
%\begin{equation}\label{eq:int1}
%\int_{\mathbb{S}^{\ell}} f(\xi) d\xi^{(\ell)} = 0, \quad \ell = 1,2,\ldots
%\end{equation}
%Hence it remains to show that
%\begin{equation}\label{eq:sum1}
%\sum_{x \in X} f(x) = 0.
%\end{equation}
%If $c_{d+1} > 0$, we obtain (\ref{eq:sum1}) by the embedding assumption.
%So we may assume that $c_{d+1} = 0$.
%Again by the embedding assumption, it follows that
%\[
%\begin{gathered}
%\sum_{x \in X} f(x) = \sum_{x \in X} x_1^{c_1} \cdots x_d^{c_d}
%x_{d+1}^{0} = \sum_{x \in X} x_1^{c_1} \cdots x_d^{c_d} \\
%= \int_{\mathbb{S}^{d-1}} f(\xi) d\xi^{(d-1)}.
%\end{gathered}
%\]
%By (\ref{eq:int1}) this equals $0$.
\end{proof}

\begin{prop}
\label{prop:sph_opt}
If \( n > 2m \) and \( n \) is odd, then
there exists a non-antipodal spherical $T_m$-design
with $n$ points.
In particular, $2m+1$ is the smallest size of a non-antipodal
spherical $T_m$-design.
\end{prop}
\begin{proof}
Let $X$ be a regular $(2m+1)$-gon in $\mathbb{S}^1$.
Then $X$ is a spherical $2m$-design (see \cite{MR679209}),
and hence a spherical $T_m$-design. 
For $n>2m+1$ and $n$ is odd, 
pick $n-2m-1$ pairs of antipodal points in $\mathbb{S}^{1}$ disjoint from $X$.
Adding these pairs of points to $X$ gives
a desired $T_m$-design with $n$ points.
Since $|X|$ is odd, $X$ is not antipodal.

Furthermore, by Lemma~\ref{lem:embedding}, a spherical $T_m$-design in $\mathbb{S}^1$,
when embedded in $\mathbb{S}^{d-1}$, is also a spherical $T_m$-design.

The second assertion then follows from Theorem~\ref{thm:sph_main}.
\end{proof}

%Proposition~\ref{prop:sph_opt} in particular shows the existence of 
%a spherical $T_m$-design with $2m+1$ points. 
%This implies that the inequality $n\le2m$ in Theorem~\ref{thm:sph_main} cannot be weakened.
In the proof of Proposition~\ref{prop:sph_opt}, we used the natural embedding
of $\mathbb{S}^{1}$ into $\mathbb{S}^{d-1}$. 
By \cite[Theorem~1.1]{MR3339051}, a more elaborate embedding 
of a regular $(2m+1)$-gon into $\mathbb{S}^{d-1}$ can be used to yield
a spherical design with harmonic index $2m$.

Although Proposition~\ref{prop:sph_opt} shows that the inequality $n\leq 2m$
in the assumption of Theorem~\ref{thm:sph_main} cannot be relaxed,
it does not assert the existence of a non-antipodal spherical $T_{m}$-design with
$n$ points for even $n>2m$.
We show that Proposition~\ref{prop:sph_opt} cannot be
strengthened
%%% strenghened
to include the case where $n$ is even, by proving the non-existence
of a non-antipodal $T_2$-design with $6$ points on $\mathbb{S}^1$.
Note that Proposition~\ref{prop:sph_opt} does imply the existence of a
non-antipodal $T_m$-design with $n\geq4m+2$ points as a union of two disjoint $(2m+1)$-gons.
\begin{prop}
There is no non-antipodal $T_2$-design on $\mathbb{S}^1$ of the cardinality $6$.
\end{prop}

\begin{proof}
Let $X\subset \mathbb{S}^1$ be a spherical $T_2$-design with $|X|=6$ points.
We will show that $X$ is antipodal. Suppose not. Then $X\cap(-X)=\emptyset$ by
Theorem~\ref{thm:sph_main}. In particular, we may assume
by rotating that $(\pm x_1,y_1)\in X$ for some $x_1$ with $0<x_1<1$, $0<y_1$.
The projection $\bar{X}$ of $X$ onto the $x$-coordinate gives a weighted interval $T_2$-design,
and by removing $\pm x_1$ with multiplicity $1$, we obtain a weighted interval $T_2$-design $X'$ consisting of $4$ points. By Proposition~\ref{prop:antipodal2}, these $4$ points form
\aiantipodal\
%%% \aiantipodal
multiset $X'$. 
Suppose first that 
$4$ points of $X'$ are not distinct. If there are only $3$ distinct points, then 
one of the points
%%% one of the point
is $0$ with multiplicity $2$, which implies $(0,\pm1)\in X$.
This contradicts the fact that $X\cap(-X)=\emptyset$.
If there are only two distinct points, then
we have $X'=\{\pm x_2\}$, which also contradicts the fact that $X\cap(-X)=\emptyset$.
%%% we have $X'=\{\pm x_2\}$. In this case, this also contradicts the fact that $X\cap(-X)=\emptyset$.
We next suppose that $X'$ consists of $4$ distinct points $\{\pm x_2,\pm x_3\}$.
Then we may assume
\[X=\{(\pm x_1,y_1),(x_2, y_2),(-x_2,\epsilon_2 y_2),
(x_3, y_3),(-x_3,\epsilon_3 y_3)\},\]
where $\epsilon_2,\epsilon_3\in\{\pm1\}$, $x_2^2+y_2^2=x_3^2+y_3^2=1$,
$y_2>y_3>0$.
Since $X\cap(-X)=\emptyset$, we have $\epsilon_2=1$ and similaly $\epsilon_3=1$.
But then $y_1+ y_2+ y_3=0$ and $y_1^3+ y_2^3+ y_3^3=0$.
This cannot occur, since $y_1y_2y_3\neq0$.
\end{proof}
%%%%%%%%%%%%%%%%%%%%%%%%%%%%%%%%%%%%%%%%%%%
%%%%%%%%%%%%%%%%%%%%%%%%%%%%%%%%%%%%%%%%%%%
%%%%%%%%%%%%%%%%%%%%%%%%%%%%%%%%%%%%%%%%%%%
%%%%%%%%%%%%%%%%%%%%%%%%%%%%%%%%%%%%%%%%%%%
\subsection*{Acknowledgements}
The authors would like to thank Eiichi Bannai, Hiroshi Nozaki, Kohji Tasaka, and Masatake Hirao for helpful discussions. The third author is supported in part by the Early Support Program for Grant-in-Aid for Scientific Research (B) of Kobe University.

\bibliographystyle{plain}
\bibliography{mms}

\end{document}